\documentclass[leqno, 12pt]{article}
\usepackage{amsmath}
\usepackage{amssymb}
\usepackage{amsthm}

\theoremstyle{plain}
\newtheorem{theorem}{\indent\sc Theorem}[section]\newtheorem{lemma}[theorem]{\indent\sc Lemma}\newtheorem{corollary}[theorem]{\indent\sc Corollary}\newtheorem{proposition}[theorem]{\indent\sc Proposition}
\theoremstyle{definition}\newtheorem{definition}[theorem]{\indent\sc Definition}

\def\address#1#2{\begingroup\noindent\parbox[t]{7.8cm}{\small{\scshape\ignorespaces#1}\par\vskip1ex\noindent\small{\itshape E-mail}\/: #2\par\vskip4ex}\hfill\endgroup}
\newcommand{\hfl}[2]{\smash{\mathop{\hbox to 9
mm{\rightarrowfill}}\limits^{\scriptstyle #1}_{\scriptstyle #2}}}

\title{\uppercase{The Gras conjecture in function fields by Euler systems}}
\author{\textsc{Hassan OUKHABA \& St\'ephane VIGUI\'E}}
\begin{document}
\maketitle
\begin{abstract}\footnote[0]{2000
\textit{Mathematics Subject Classification. Primary 11R58, Secondary 11G09}}
\footnote[1]{\textit Key words. Function fields, Stark units, Gras conjecture, Euler systems, Drinfel'd modules}
We use Euler systems to prove the Gras conjecture for groups generated by Stark units in global function fields. The techniques applied here are classical and go back to Thaine, Kolyvagin and Rubin. We obtain our Euler systems from the torsion points of sign-normalized Drinfel'd modules.
\end{abstract}
\bibliographystyle{plain}
\section{Introduction}\label{intro}

Let $k$ be a global function field of characteristic a prime number $\rho$. Let $\mathbb{F}_q$, $q:=\rho^n$, be the field of constants of $k$. Let $\infty$ be a place of $k$ and let $\mathcal{O}_k$ be the Dedekind ring of functions $f\in k$ regular outside $\infty$. We denote by $k_\infty$ the completion of $k$ at $\infty$. Let us also fix $K\subset k_\infty$ a finite abelian extension of $k$, and write $\mathrm{G}$ for the Galois group of $K/k$, $\mathrm{G}:=\mathrm{Gal}(K/k)$. The inclusion $K\subset k_\infty$ simply means that the place $\infty$ splits completely in $K$. Let $\mathcal{O}_K$ and $\mathcal{O}_K^\times$ be respectively the integral closure of $\mathcal{O}_k$ in $K$ and the group of units of $\mathcal{O}_K$. One may use Stark units to define a subgroup $\mathcal{E}_K$ of $\mathcal{O}_K^\times$ such that the factor group $\mathcal{O}_K^\times/\mathcal{E}_K$ is finite (See the definition of $\mathcal{E}_K$ in the next section). Let $H\subset k_\infty$ be the maximal abelian unramified extension of $k$ contained in $k_\infty$. In \cite{ouk91} it is proved that in case $K\subset H$ we have
\begin{equation}\label{indice1}
[\mathcal{O}_K^\times:\mathcal{E}_K]=\frac{h(\mathcal{O}_K)}{[H:K]},
\end{equation} 
where $h(\mathcal{O}_K)$ is the order of the ideal class group of $\mathcal{O}_K$. In the general case, one may obtain (a complicated formula for) the quotient $[\mathcal{O}_K^\times:\mathcal{E}_K]/h(\mathcal{O}_K)$ in terms of numerical invariants of $K/k$, exactly as the first author did in \cite[formula (3.3)]{ouk09}. Let $\mathrm{Cl}(\mathcal{O}_K)$ be the ideal class group of $\mathcal{O}_K$. Recently, in \cite{Vig}, the second author used his notion of index-module to prove the following remarkable result. Let $\mathrm{g}:=[K:k]$, then, for every nontrivial irreducible rational character $\Psi$ of $\mathrm{G}$ we have
\begin{equation}\label{indice2}
\bigl[e_{\Psi}\bigl(\mathbb{Z}[\mathrm{g}^{-1}]\otimes_{\mathbb{Z}}\mathcal{O}_K^\times\bigr) :e_{\Psi}\bigl(\mathbb{Z}[\mathrm{g}^{-1}]\otimes_{\mathbb{Z}}\mathcal{E}_K\bigr)\bigr]=\#[e_{\Psi}\bigl(\mathbb{Z}[\mathrm{g}^{-1}]\otimes_{\mathbb{Z}}\mathrm{Cl}(\mathcal{O}_K)\bigr)],
\end{equation}  
where $e_\Psi$ is the idempotent of $\mathbb{Z}[\mathrm{g}^{-1}][\mathrm{G}]$ associated to $\Psi$. By $\# X$ we mean the cardinality of the finite set $X$. The formula (\ref{indice2}) may be considered as a weak form of the Gras conjecture for $\mathcal{E}_K$.\par
In this paper we use Euler systems to prove the Gras conjecture for $\mathcal{E}_K$, for every prime number $p$, $p\nmid\rho[K:k]$ and every irreducible $\mathbb{Z}_p$-character of $\mathrm{G}$, not in the set $\Xi_p$ defined as follows. If $p\neq\rho$ is a prime number then we denote by $\mu_p$ the group of $p$-th roots of unity. If $p\,\vert\,[H:k]$ and $\mu_p\subset K$ then we denote by $\omega$ the Teichmuller character giving the action of $\mathrm{G}$ on $\mu_p$. Let $f$ be the order of $q$ in $(\mathbb{Z}/p\,\mathbb{Z})^\times$. Then $\mathbb{F}_q(\mu_p)=\mathbb{F}_{q^f}$ and the order of $\omega$ is $f$. We define
\begin{equation*}
\Xi_p:=
\begin{cases}
\emptyset & \mathrm{if}\ \mu_p\not\subset K\ \mathrm{or}\ p\nmid[H:k],\\
\{\omega^i,\ i\in(\mathbb{Z}/f\mathbb{Z})^\times\} & \mathrm{if}\ p\,\vert\,[H:k]\ \mathrm{and}\ \mu_p\subset K.
\end{cases}
\end{equation*}  
\begin{theorem}\label{tresgras} Let $p$ be a prime number such that $p\nmid \rho[K:k]$. Let $\chi$ be a nontrivial irreducible $\mathbb{Z}_p$-character of $\mathrm{G}$ such that $\chi\not\in\Xi_p$. Then we have
\begin{equation}\label{gras}
\bigl[e_{\chi}\bigl(\mathbb{Z}_p\otimes_{\mathbb{Z}}\mathcal{O}_K^\times\bigr) :e_{\chi}\bigl(\mathbb{Z}_p\otimes_{\mathbb{Z}}\mathcal{E}_K\bigr)\bigr]=\#[e_{\chi}\bigl(\mathbb{Z}_p\otimes_{\mathbb{Z}}\mathrm{Cl}(\mathcal{O}_K)\bigr)],
\end{equation}
where $e_\chi$ is the idempotent of $\mathbb{Z}_p[\mathrm{G}]$ associated to $\chi$.
\end{theorem}
The proof of this theorem is given at the end of the paper. The formula (\ref{gras}) is proved first by Keqin Feng and Fei Xu in \cite{Feng96} when  $k=\mathbb{F}_q(T)$ is a rational function field in one variable, $\infty$ is the place associated to the unique pole of $(1/T)$, $K=H_{\mathfrak{m}}$ for some ideal $\mathfrak{m}$ of $\mathcal{O}_k=\mathbb{F}_q[T]$ (the field $H_{\mathfrak{m}}$ is defined below) and $p\nmid q(q-1)[K:k]$. To obtain their result, Keqin Feng and Fei Xu also used the method of Euler systems. 
\section{The group $\mathcal{E}_K$}\label{the group}
For each nonzero ideal $\mathfrak{m}$ of $\mathcal{O}_k$, we denote by $H_{\mathfrak{m}}$ the maximal abelian extension of $k$ contained in $k_\infty$, such that the conductor of $H_{\mathfrak{m}}/k$ divides $\mathfrak{m}$. The function field version of the abelian conjectures of Stark, proved by P.\,Deligne in \cite{Tate84} by using \'etale cohomology or by D.\,Hayes in \cite[Theorem 1.1]{Hay85} by using Drinfel'd modules, asserts, for any $\mathfrak{m}\not\in\{(0),\mathcal{O}_k\}$, the existence of an element $\varepsilon=\varepsilon_{\mathfrak{m}}\in H_{\mathfrak{m}}$, unique up to a root of unity such that
\begin{itemize}
\item[(i)] If we set $w_\infty=q^{d_\infty}-1$, where $d_\infty$ is the degree of $\infty$,
then the extension $H_{\mathfrak{m}}(\varepsilon^{1/w_\infty})/k$ is abelian.
\item[(ii)] If $\mathfrak{m}$ is divisible by two prime ideals then
$\varepsilon$ is a unit of $\mathcal{O}_{H_{\mathfrak{m}}}$. If $\mathfrak{m}=\mathfrak{q}^e$, where $\mathfrak{q}$ is a prime ideal then 
\begin{equation*}
\varepsilon\mathcal{O}_{H_{\mathfrak{m}}}=(\mathfrak{q})_{\mathfrak{m}}^{\frac{w_\infty}{w_k}}\end{equation*}
where $w_k:=q-1$ and $(\mathfrak{q})_{\mathfrak{m}}$ is the product of the prime ideals of $\mathcal{O}_{H_{\mathfrak{m}}}$ which divide $\mathfrak{q}$. 
\item[(iii)] We have
\begin{equation}\label{kronecker}
L_{\mathfrak{m}}(0,\chi)=
\frac{1}{w_\infty}\sum_{\sigma\in\mathrm{Gal}(H_{\mathfrak{m}}/k)
}\chi(\sigma)v_\infty(\varepsilon^\sigma),
\end{equation}
for all complex irreducible characters $\chi$ of $\mathrm{Gal}(H_{\mathfrak{m}}/k)$.
\end{itemize}
Here $s\longmapsto L_{\mathfrak{m}}(s,\chi)$ is the
$L$-function associated to $\chi$, defined for the complex numbers
$s$ such that $Re(s)>1$, by the Euler product
\begin{equation*}
L_{\mathfrak{m}}(s,\chi)=
\prod_{\mathfrak{v}\nmid\mathfrak{m}}\bigl(1-\chi(\sigma_{\mathfrak{v}})N(\mathfrak{v})^{-s}\bigr)^{-1},
\end{equation*}
where $\mathfrak{v}$ runs through all the places of $k$ not dividing $\mathfrak{m}$. For such a place, $\sigma_{\mathfrak{v}}$ and $N(\mathfrak{v})$ are the Frobenius automorphism of $H_{\mathfrak{m}}/k$ and the order of the residue field at $\mathfrak{v}$ respectively. Let us remark that $\sigma_{\infty}=1$ and $N(\infty)=q^{d_\infty}$.\par
For any finite abelian extension $L$ of $k$, we denote by $\mu_L$ the group of roots of unity (nonzero constants) in $L$, by $w_L$ the order of $\mu_L$ and by $\mathcal{F}_L\subset\mathbb{Z}[\mathrm{Gal}(L/k)]$ the annihilator of $\mu_L$. The description of $\mathcal{F}_L$ given in \cite[Lemma 2.5]{Hay85} and the property (i) of $\varepsilon_{\mathfrak{m}}$ imply, in particular, that for any $\eta\in\mathcal{F}_{H_{\mathfrak{m}}}$ there exists $\varepsilon_{\mathfrak{m}}(\eta)\in H_{\mathfrak{m}}$ such that 
\begin{equation*}
\varepsilon_{\mathfrak{m}}(\eta)^{w_\infty}=\varepsilon_{\mathfrak{m}}^\eta.
\end{equation*}
\begin{definition}\label{ek} Let $\mathcal{P}_K$ be the subgroup of $K^\times$ generated by $\mu_K$ and by all the norms 
\begin{equation*}
N_{H_{\mathfrak{m}}/H_{\mathfrak{m}}\cap K}(\varepsilon_{\mathfrak{m}}(\eta)),
\end{equation*}
where $\mathfrak{m}$ is any nonzero proper ideal of $\mathcal{O}_k$ and $\eta$ is any element of $\mathcal{F}_{H_{\mathfrak{m}}}$. We define 
\begin{equation*}
\mathcal{E}_K:=\mathcal{P}_K\cap\mathcal{O}_K^\times.
\end{equation*}
\end{definition}
\section{The Euler system}
For any finite abelian extension $F$ of $k$, and any fractional ideal $\mathfrak{a}$ of $\mathcal{O}_k$ prime to the conductor of $F/k$, we denote by $(\mathfrak{a}, F/k)$ the automorphism of $F/k$ associated to $\mathfrak{a}$ by the Artin map. If $\mathfrak{a}\subset\mathcal{O}_k$ then we denote by $N(\mathfrak{a})$ the cardinality of $\mathcal{O}_k/\mathfrak{a}$. Let $\mathcal{I}(\mathcal{O}_k)$ be the group of fractional ideals of $\mathcal{O}_k$ and let us consider its subgroup
$\mathcal{P}(\mathcal{O}_k):=\{x\mathcal{O}_k,\ x\in k^\times\}$. Then, the Artin map gives an isomorphism from ${\bf Pic}(\mathcal{O}_k):=\mathcal{I}(\mathcal{O}_k)/\mathcal{P}(\mathcal{O}_k)$ into Gal$(H/k)$. Let $p$ be a prime number, and let $\mathbf{Pic}_p(\mathcal{O}_k)$ be the $p$-part of ${\bf Pic}(\mathcal{O}_k)$. Then, fix $\mathfrak{a}_1,\ldots,\mathfrak{a}_s$, a finite set of ideals of $\mathcal{O}_k$ such that
\begin{equation}\label{pic}
{\bf Pic}_p(\mathcal{O}_k)=<\bar{\mathfrak{a}}_1>\times\cdots\times<\bar{\mathfrak{a}}_s>,
\end{equation}
where $<\bar{\mathfrak{a}}_i>\neq1$ is the group generated by the class $\bar{\mathfrak{a}}_i$ of $\mathfrak{a}_i$ in ${\bf Pic}(\mathcal{O}_k)$. If $n_i$ is the order of $<\bar{\mathfrak{a}}_i>$, then $(\mathfrak{a}_i)^{n_i}=a_i\mathcal{O}_k$, with $a_i\in\mathcal{O}_k$. If ${\bf Pic}_p(\mathcal{O}_k)=1$ then we set $s=1$ and $\mathfrak{a}_1:=\mathcal{O}_k$ and $a_1=1$.\par  
Let $p\neq\rho$ be a prime number, and let $M$ be a power of $p$. Let $\mu_M$ be the group of $M$-th roots of unity. Then we define
\begin{equation}\label{kaem}
K_M:=
\begin{cases}
K((\mathbb{F}_q^\times)^{1/M}) & \mathrm{if}\ \mu_p\subset k\\
K(\mu_M) & \mathrm{if}\ \mu_p\not\subset k.
\end{cases}
\end{equation}
Moreover, we denote by $\mathcal{L}$ the set of prime ideals $\ell$ of $\mathcal{O}_k$ such that $\ell$ splits completely in the Galois extension $K_M\bigl(a_1^{1/M},\ldots,a_s^{1/M}\bigr)/k$.
Exactly as in \cite[Lemma 3]{Ru94} we have
\begin{lemma}\label{extension} For each prime $\ell\in\mathcal{L}$ there exists a cyclic extension $K(\ell)$ of $K$ of degree $M$, contained in the compositum $K.H_{\ell}$, unramified outside $\ell$, and such that $K(\ell)/K$ is totally ramified at all primes above $\ell$.
\end{lemma}
\begin{proof} Let us remark that the group $C:=\mathrm{Gal}(H_\ell/H)$ is cyclic of order $(N(\ell)-1)/w_k$. Since  $\ell$ splits completely in $K_M$ the integer $M$ divides $(N(\ell)-1)/w_k$. In particular the fixed field of $C^M$ is a cyclic extension of $H$ of degree $M$. Let us denote it by $H(\ell)$. Let $\sigma_{\mathfrak{a}_i}:=(\mathfrak{a}_i, H(\ell)/k)$, for $i=1,\ldots,s$ (remark that $\mathfrak{a}_i$ is prime to $\ell$). Let $D:=<\sigma_{\mathfrak{a}_1},\ldots,\sigma_{\mathfrak{a}_s}>$ be the subgroup of $\mathrm{Gal}(H(\ell)/k)$ generated by the automorphisms $\sigma_{\mathfrak{a}_i}$. Let $E$ be the fixed field of $D$. Let $P$ be the $p$-part of $\mathrm{Gal}(H/k)$ and $L$ be the fixed field of $P$. From (\ref{pic}) we deduce that $E\cap H=L$. Moreover, if $\sigma\in D\cap\mathrm{Gal}(H(\ell)/H)$ then $\sigma$ is the restriction to $H(\ell)$ of  some automorphism $(x\mathcal{O}_k, H_\ell/k)$, where $x=\prod_{i=1}^s a_i^{e_i}$. But such elements are $M$-th powers modulo $\ell$ for $\ell\in\mathcal{L}$, which implies that $(x\mathcal{O}_k, H_\ell/k)\in C^M$ and hence $\sigma=1$. Therefore $H(\ell)=E.H$. It is obvious now that $E=E'L$, where $E'$ is a subfield of $E$ such that $E'\cap L=k$. The field $K(\ell):=E'.K$ satisfies the required properties stated in the lemma.
\end{proof}
Let $\mathcal{S}$ be the set of squarefree ideals of $\mathcal{O}_k$ divisible only by primes $\ell\in\mathcal{L}$. If $\mathfrak{a}=\ell_1\cdots\ell_n\in\mathcal{S}$ then we set $K(\mathfrak{a}):=K(\ell_1)\cdots K(\ell_n)$ and $K(\mathcal{O}_k):=K$. If $\mathfrak{g}$ is an ideal of $\mathcal{O}_k$ then we denote by $\mathcal{S}(\mathfrak{g})$ the set of ideals $\mathfrak{a}\in\mathcal{S}$ that are prime to $\mathfrak{g}$. Following Rubin we define an Euler system to be a function
\begin{equation*}
\alpha:\mathcal{S}(\mathfrak{g})\longrightarrow k_\infty^\times,
\end{equation*}
such that
\begin{description}
\item{E1.}\ $\alpha(\mathfrak{a})\in K(\mathfrak{a})^\times$.
\item{E2.}\ $\alpha(\mathfrak{a})\in\mathcal{O}_{K(\mathfrak{a})}^\times$, if $\mathfrak{a}\neq\mathcal{O}_k$.
\item{E3.}\  $N_{K(\mathfrak{a}\ell)/K(\mathfrak{a})}\bigl(\alpha(\mathfrak{a}\ell)\bigr)=\alpha(\mathfrak{a})^{1-\mathrm{Fr}(\ell)^{-1}}$, where Fr$(\ell)$ is the Frobenius of $\ell$ in Gal$(K(\mathfrak{a})/k)$.
\item{E4.}\ $\alpha(\mathfrak{a}\ell)\equiv\alpha(\mathfrak{a})^{\mathrm{Frob}(\ell)^{-1}(N(\ell)-1)/M}$ modulo all primes above $\ell$.
\end{description}\par
We use the theory of sign-normalized Drinfel'd modules, developped by D.\,Hayes in \cite{Hay85}, to produce Euler systems. Let $\Omega_k$ be the completion of the algebraic closure of $k_\infty$. Then $\Omega_k$ is algebraically closed. We briefly recall the definition of the Drinfel'd module $\Phi^\Gamma$, associated to any $\mathcal{O}_k$-lattice $\Gamma$ of $\Omega_k$, that is, any finitely generated $\mathcal{O}_k$-submodule of $\Omega_k$, of rank one. Let $\Omega_k[\mathbf{F}]$ be the left twisted polynomial ring in the Frobenius endomorphism $\mathbf{F}:x\longmapsto x^q$, with the rule $\mathbf{F}.w=w^q.\mathbf{F}$, for all $w\in\Omega_k$. Then, $\Phi^\Gamma: \mathcal{O}_k\longrightarrow\Omega_k[\mathbf{F}]$, is the $\mathbb{F}_q$-algebra homomorphism such that the image of $x$ is the unique element $\Phi_x^\Gamma$ of $\Omega_k[\mathbf{F}]$ satisfying
\begin{equation}\label{exponent}
e_\Gamma(xz)=\Phi_x^\Gamma(e_\Gamma(z)),\ \mathrm{for\ all}\ z\in\Omega_k.
\end{equation}
Here, by $e_\Gamma(z)$ we mean the infinite product
\begin{equation*}
e_\Gamma(z):=z\prod_{\gamma\in\Gamma}(1-z/\gamma)\quad(\gamma\neq0).
\end{equation*} 
\begin{theorem} The infinite product $e_\Gamma(z)$ converges uniformly on any bounded subset of $\Omega_k$. We thus obtain a surjective $\mathbb{F}_q$-linear entire function of $\Omega_k$, periodic with $\Gamma$ as a group of periods.
\end{theorem}
\begin{proof} See \cite[Theorem 8.5]{Hay92}.
\end{proof}
\begin{theorem} If $\Gamma\subset\Gamma'$ then the factor group $\Gamma'/\Gamma$ is finite. Moreover,
\begin{equation}\label{rela1}
e_{\Gamma'}(z)=P(\Gamma,\Gamma';e_\Gamma(z)),\ \mathrm{for\ all}\ z\in\Omega_k,
\end{equation} 
where $P(\Gamma,\Gamma';t)$ is the polynomial
\begin{equation*}
P(\Gamma,\Gamma';t):=t\prod_{\gamma\neq0}(1-t/\gamma),\quad\quad(\gamma\in e_\Gamma(\Gamma')).
\end{equation*}
\end{theorem}
\begin{proof} Since $\Gamma$ and $\Gamma'$ are finitely generated $\mathcal{O}_k$-submodules of $\Omega_k$, of rank one there exist $\alpha,\alpha'\in k^\times$ and fractional ideals $\mathfrak{a}$ and $\mathfrak{a}'$ of $k$ such that $\Gamma=\alpha\mathfrak{a}$ and $\Gamma'=\alpha'\mathfrak{a}'$. The hypothesis $\Gamma\subset\Gamma'$ means that $\mathfrak{c}:=\alpha\alpha'^{-1}\mathfrak{a}\mathfrak{a}'^{-1}$ is a nonzero ideal of $\mathcal{O}_k$. But $\Gamma\subset\Gamma'$ is isomorphic to $\mathcal{O}_k/\mathfrak{c}$ which is necessarily finite. The other assertion of the theorem corresponds to \cite[Theorem 8.7]{Hay92}.
\end{proof}
It is easy to check the following equality
\begin{equation}\label{lineaire}
P(\omega\Gamma,\omega\Gamma';\omega t)=\omega P(\Gamma,\Gamma';t),
\end{equation}
for all $\omega\in\Omega_k^\times$.
\begin{corollary} Let $\Gamma\subset\Gamma'\subset\Gamma''$ be three $\mathcal{O}_k$-lattices of $\Omega_k$. Then
\begin{equation}\label{composition}
P(\Gamma,\Gamma'';t)=P(\Gamma',\Gamma'';P(\Gamma,\Gamma';t)).
\end{equation}
\end{corollary}
\begin{proof} The identity (\ref{composition}) is an immediate consequence of (\ref{rela1}).
\end{proof}
Let us set
\begin{equation*}
\delta(\Gamma,\Gamma'):=\prod_{\gamma\neq0}\gamma^{-1},\quad\quad(\gamma\in e_\Gamma(\Gamma')).
\end{equation*}
\begin{lemma}\label{fix} We have 
\begin{equation*}
\Phi_x^\Gamma(t)=xP(\Gamma,x^{-1}\Gamma;t),
\end{equation*}
for all $x\in\mathcal{O}_k-\{0\}$. In particular, the leading coefficient of $\Phi_x^\Gamma$ is $x\delta(\Gamma,x^{-1}\Gamma)$.
\end{lemma}
\begin{proof} We have to apply (\ref{rela1}) to the lattice $\Gamma':=x^{-1}\Gamma$. But since $e_\Gamma(xz)=xe_{\Gamma'}(z)$ we obtain $e_\Gamma(xz)=xP(\Gamma,x^{-1}\Gamma;e_\Gamma(z))$. Comparing with  (\ref{exponent}) gives us the formula for $\Phi_x^\Gamma(t)$.  
\end{proof}
\begin{proposition}\label{fia} Let $\mathfrak{a}$ be a nonzero ideal of $\mathcal{O}_k$. Then, the left ideal of $\Omega_k[\mathbf{F}]$ generated by $\Phi_a^\Gamma$, $a\in\mathfrak{a}$, is principal generated by the monic polynomial 
\begin{equation*}
\Phi_{\mathfrak{a}}^\Gamma(t):=\delta(\Gamma,\mathfrak{a}^{-1}\Gamma)^{-1}P(\Gamma,\mathfrak{a}^{-1}\Gamma;t).
\end{equation*}
\end{proposition}
\begin{proof} If $a\in\mathfrak{a}$ then $\Phi_a^\Gamma(t)=aP(\mathfrak{a}^{-1}\Gamma,a^{-1}\Gamma;P(\Gamma,\mathfrak{a}^{-1}\Gamma;t))$, thanks to (\ref{composition}). This proves that our left ideal is generated by $P(\Gamma,\mathfrak{a}^{-1}\Gamma;t)$. 
\end{proof}
\begin{corollary} Let $D:\Omega_k[\mathbf{F}]\longrightarrow\Omega_k$ be the map which associates to a polynomial in $\mathbf{F}$ its constant term. Then 
\begin{equation}\label{identity}
D(\Phi_x^\Gamma)=x\quad\mathrm{and}\quad  D(\Phi_{\mathfrak{a}}^\Gamma)=\delta(\Gamma,\mathfrak{a}^{-1}\Gamma)^{-1}.
\end{equation}
Moreover, if $s_{\Phi^\Gamma}(x)$ is the leading coefficient of $\Phi_x^\Gamma$. Then
\begin{equation}\label{pieuvre}
D(\Phi_{x\mathcal{O}_k}^\Gamma)=s_{\Phi^\Gamma}(x)^{-1}x.
\end{equation}
\end{corollary}
\begin{proof} This is immediate from the explicite description of $\Phi_x^\Gamma$ and $\Phi_{\mathfrak{a}}^\Gamma$ given by Lemma \ref{fix} and Proposition \ref{fia}.
\end{proof}
Let $k(\infty)=\mathbb{F}_{q^{d_\infty}}$ be the field of constants of $k_\infty$ and let $\mathbf{sgn}$ be a sign-function of $k_\infty$ (fixed throughout this article), that is a continuous group homomorphism $\mathbf{sgn}:k_\infty^\times\longrightarrow k(\infty)^\times$ such that $\mathbf{sgn}(x)=x$ for all $x\in k(\infty)^\times$. 
\begin{definition}The Drinfel'd module $\Phi^\Gamma$ is called $\mathbf{sgn}$-normalized if the leading coefficient of $\Phi_x^\Gamma$ has the form $\mathbf{sgn}(x)^{\tau_\Gamma}$, where $\tau_\Gamma\in\mathrm{Gal}(k(\infty)/\mathbb{F}_q)$ is an automorphism of $k(\infty)/\mathbb{F}_q$ depending only on $\Gamma$.
\end{definition}
\begin{theorem}Let $\mathfrak{c}$ be a fractional ideal of $\mathcal{O}_k$. Then, there exists a nonzero element $\xi(\mathfrak{c})\in\Omega_k^\times$ such that the Drinfel'd module $\Phi^{\tilde{\mathfrak{c}}}$ associated to $\tilde{\mathfrak{c}}:=\xi(\mathfrak{c})\mathfrak{c}$ is a $\mathbf{sgn}$-normalized $\mathcal{O}_k$-module. Moreover, $\xi(\mathfrak{c})$ is determined up to multiplication by elements of $k(\infty)^\times$.
\end{theorem}
\begin{proof} We refer the reader to \cite[Theorem 12.3 and Proposition 13.1]{Hay92}.
\end{proof}

\begin{proposition} If $\xi(\mathcal{O}_k)$ is fixed then for every fractional ideal $\mathfrak{c}$ of $\mathcal{O}_k$ there exists a unique choice of $\xi(\mathfrak{c})$ so that
\begin{equation}\label{cond}
\xi(\mathfrak{a}^{-1}\mathfrak{c})=D(\Phi_{\mathfrak{a}}^{\tilde{\mathfrak{c}}})\xi(\mathfrak{c}),
\end{equation}   
for all ideal $\mathfrak{a}$ of $\mathcal{O}_k$. 
\end{proposition}
\begin{proof} By \cite[Theorem 8.14 and Theorem 13.8]{Hay92} if $\xi(\mathfrak{c})$ is given then $\xi(\mathfrak{a}^{-1}\mathfrak{c})$ is determined by (\ref{cond}) up to an element of $k(\infty)^\times$. We leave it to the reader to check that if $\xi(\mathcal{O}_k)$ is fixed then (\ref{cond}) allows us to fix a value of $\xi(\mathfrak{c})$ for every $\mathfrak{c}$.
\end{proof}
Let us fix $\xi(\mathcal{O}_k)$ and let us denote by $H^*_{\mathfrak{e}}$ the normalizing field with respect to $\mathbf{sgn}$, cf. \cite[Definition 4.9]{Hay85}. We recall that if $\Phi$ is a $\mathbf{sgn}$-normalized Drinfel'd module, then $H^*_{\mathfrak{e}}$ is the subfield of $\Omega_k$ generated by the coefficients of the polynomials $\Phi_x$, $x\in\mathcal{O}_k$. 
\begin{theorem}The extension $H^*_{\mathfrak{e}}/k$ is finite abelian and unramified except at $\infty$. The ramification index at $\infty$ is $w_\infty/w_k$. Also we have $H\subset H^*_{\mathfrak{e}}$ and $[H^*_{\mathfrak{e}}:H]=w_\infty/w_k$.
\end{theorem}
\begin{proof} See \cite[Theorem 4.10]{Hay85} or \cite[\S14]{Hay92}.
\end{proof}
Let $\mathfrak{m}$ be a nonzero proper ideal of $\mathcal{O}_k$, and let us consider the element
\begin{equation*}
\lambda_{\mathfrak{m}}:=\xi(\mathfrak{m})e_{\mathfrak{m}}(1).
\end{equation*}
Then, we may deduce from above the following
\begin{lemma} We have
\begin{equation}\label{rela3}
\xi(\mathfrak{a}^{-1}\mathfrak{m})e_{\mathfrak{a}^{-1}\mathfrak{m}}(1)=\Phi_{\mathfrak{a}}^{\tilde{\mathfrak{m}}}(\lambda_{\mathfrak{m}}),
\end{equation}
for any nonzero ideal $\mathfrak{a}$ of $\mathcal{O}_k$.
\end{lemma}
\begin{proof} By (\ref{cond}), (\ref{rela1}) and (\ref{lineaire}) we obtain $\xi(\mathfrak{a}^{-1}\mathfrak{m})e_{\mathfrak{a}^{-1}\mathfrak{m}}(1)=D(\Phi_{\mathfrak{a}}^{\tilde{\mathfrak{m}}})P(\tilde{\mathfrak{m}},\mathfrak{a}^{-1}\tilde{\mathfrak{m}};\lambda_{\mathfrak{m}})$. Now use Proposition \ref{fia} and (\ref{identity}) to conclude.
\end{proof}
In the sequel we shall use the field $k_{\mathfrak{m}}:=H^*_{\mathfrak{e}}(\lambda_{\mathfrak{m}})$. As proved in \cite[\S4]{Hay85} $k_{\mathfrak{m}}$ is a finite abelian extension of $k$. Moreover, $H_{\mathfrak{m}}\subset k_{\mathfrak{m}}$ and
\begin{theorem}
\begin{equation}\label{stark1}
N_{k_{\mathfrak{m}}/H_{\mathfrak{m}}}(\lambda_{\mathfrak{m}})=-\lambda_{\mathfrak{m}}^{w_\infty}\quad\mathrm{is\ a\  Stark\ unit}
\end{equation}
\end{theorem}
\begin{proof} The equality is Theorem 4.17 of \cite{Hay85}. The fact that $-\lambda_{\mathfrak{m}}^{w_\infty}$ is a Stark unit is also proved in \cite[\S\S 4 and 6]{Hay85}. 
\end{proof}
\begin{lemma}\label{unlemme} If $\mathfrak{q}$ is a prime ideal of $\mathcal{O}_k$ then
\begin{equation}\label{distribution}
N_{k_{\mathfrak{mq}}/k_{\mathfrak{m}}}(\lambda_{\mathfrak{mq}})=\left\lbrace
\begin{array}{cc}
\lambda_{\mathfrak{m}} & \mathrm{if}\ \mathfrak{q}\vert\mathfrak{m}\\
\lambda_{\mathfrak{m}}^{1-\mathrm{Fr}(\mathfrak{q})^{-1}} & \mathrm{if}\ \mathfrak{q}\nmid\mathfrak{m}
\end{array}
\right.
\end{equation}  
where $\mathrm{Fr}(\mathfrak{q})$ is the Frobenius of $\mathfrak{q}$ in $\mathrm{Gal}(k_{\mathfrak{m}}/k)$.
\end{lemma}
\begin{proof} Let $\Gamma:=\xi(\mathfrak{mq})\mathfrak{mq}$, $\Phi:=\Phi^\Gamma$ and $\xi:=\xi(\mathfrak{mq})$. Let $X$ be a complete set of representatives modulo $\mathfrak{mq}$ of the kernel of the natural map
\begin{equation*}
(\mathcal{O}_k/\mathfrak{mq})^\times\longrightarrow(\mathcal{O}_k/\mathfrak{m})^\times.
\end{equation*}
We may choose $X$ so that $\mathbf{sgn}(x)=1$ for all $x\in X$. By \cite[formula (4.8)]{Hay85} we see that $\mathrm{Gal}(k_{\mathfrak{mq}}/k_{\mathfrak{m}})$ is equal to the set $\{(x\mathcal{O}_k,k_{\mathfrak{mq}}/k)$, $x\in X\}$. Since $\lambda_{\mathfrak{mq}}=e_\Gamma(\xi)$ and $\Phi_x=\Phi_{x\mathcal{O}_k}$ if $\mathbf{sgn}(x)=1$, the theorem 4.12 of \cite{Hay85} and the above formula (\ref{exponent}) give
\begin{equation*}
N_{k_{\mathfrak{mq}}/k_{\mathfrak{m}}}(\lambda_{\mathfrak{mq}})=\prod_{x\in X}e_\Gamma(x\xi).
\end{equation*}
Suppose for the moment that $\mathfrak{q}\vert\mathfrak{m}$. Then, the set $Y:=\{\xi t,\ 1-t\in X\}$ is a complete system of representatives of $\mathfrak{q}^{-1}\Gamma/\Gamma$. This allows us to deduce
\begin{equation*}
\prod_{x\in X}e_\Gamma(x\xi)=\prod_{y\in Y}e_\Gamma(\xi-y)=e_{\mathfrak{q}^{-1}\Gamma}(\xi)D(\Phi_{\mathfrak{q}}),
\end{equation*}
where the last equality is a direct application of (\ref{rela1}) and (\ref{identity}). But, on one hand, it is obvious that $e_{\mathfrak{q}^{-1}\Gamma}(\xi)=\xi e_{\mathfrak{m}}(1)$ and on the other hand $\xi D(\Phi_{\mathfrak{q}})=\xi(\mathfrak{m})$ by (\ref{cond}). Thus we proved the norm formula when $\mathfrak{q}\vert\mathfrak{m}$.\par
Let us now assume that $\mathfrak{q}\nmid\mathfrak{m}$.
Let $t_0\in\mathfrak{m}$ be such that $t_0\equiv1$ modulo $\mathfrak{q}$, then, the set $Z:=\{\xi t_0\}\cup\{\xi t, 1-t\in X\}$ give a complete system of representatives of $\mathfrak{q}^{-1}\Gamma/\Gamma$. Therefore (\ref{rela1}) implies
\begin{equation*}
e_\Gamma(\xi-\xi t_0)\prod_{x\in X}e_\Gamma(x\xi)=\prod_{z\in Z}e_\Gamma(\xi-z)=D(\Phi_{\mathfrak{q}})e_{\mathfrak{q}^{-1}\Gamma}(\xi)=\lambda_{\mathfrak{m}}.
\end{equation*}
Thus we have
\begin{equation*}
N_{k_{\mathfrak{mq}}/k_{\mathfrak{m}}}(\lambda_{\mathfrak{mq}})=\frac{\lambda_{\mathfrak{m}}}{\xi(\mathfrak{mq})e_{\mathfrak{mq}}(x_0)},
\end{equation*}
where $x_0=1-t_0$. But, we may choose $x_0$ such that $\mathbf{sgn}(x_0)=1$. In particular, $D(\Phi_{x_0\mathcal{O}_k})=x_0$ and $\xi(\mathfrak{mq})e_{\mathfrak{mq}}(x_0)=\xi(x_0^{-1}\mathfrak{mq})e_{x_0^{-1}\mathfrak{mq}}(1)=\xi(\mathfrak{a}^{-1}\mathfrak{m})e_{\mathfrak{a}^{-1}\mathfrak{m}}(1)$, where $\mathfrak{a}:=x_0\mathfrak{q}^{-1}$. By (\ref{rela3}), \cite[Theorem 4.12]{Hay85} and the fact that $(x_0\mathcal{O}_k, k_{\mathfrak{m}}/k)=1$ we obtain 
\begin{equation}\label{asavoir}
\xi(\mathfrak{mq})e_{\mathfrak{mq}}(x_0)=\lambda_{\mathfrak{m}}^{\mathrm{Fr}(\mathfrak{q})^{-1}}.
\end{equation}
This completes the proof of the lemma.
\end{proof}
\begin{lemma}\label{poisson} Assume $\mathfrak{q}\nmid\mathfrak{m}$. Then,
\begin{equation}\label{congruence}
\lambda_{\mathfrak{mq}}-\lambda_{\mathfrak{m}}^{\mathrm{Fr}(\mathfrak{q})^{-1}}=\lambda_{\mathfrak{q}}^{\sigma_{\mathfrak{m}}^{-1}}.
\end{equation}
where $\sigma_{\mathfrak{m}}:=(\mathfrak{m}, k_{\mathfrak{q}}/k)$ is the automorphism of $k_{\mathfrak{q}}/k$ associated to $\mathfrak{m}$ by the Artin map.
\end{lemma} 
\begin{proof} let us keep the notation of the proof of Lemma \ref{unlemme}. Then, an easy computation based on (\ref{asavoir}) gives
\begin{equation}
\lambda_{\mathfrak{mq}}-\lambda_{\mathfrak{m}}^{\mathrm{Fr}(\mathfrak{q})^{-1}}=\xi(\mathfrak{mq})e_{\mathfrak{mq}}(t_0).
\end{equation}
On the other hand, (\ref{pieuvre}) and (\ref{cond}) give $\xi(t_0^{-1}\mathfrak{mq})=s_\Phi(t_0)^{-1}t_0\xi(\mathfrak{mq})$. Thus, if we set $\mathfrak{a}:=t_0\mathfrak{m}^{-1}\subset\mathcal{O}_k$, then we obtain
\begin{equation*}
\xi(\mathfrak{mq})e_{\mathfrak{mq}}(t_0)=s_\Phi(t_0)\xi(\mathfrak{a}^{-1}\mathfrak{q})e_{\mathfrak{a}^{-1}\mathfrak{q}}(1)=s_\Phi(t_0)\Phi_{\mathfrak{a}}^{\tilde{\mathfrak{q}}}(\lambda_{\mathfrak{q}}).
\end{equation*}
The last equality is a special case of (\ref{rela3}). By \cite[Theorem 4.12 and Corollary 4.14]{Hay85} we have $\Phi_{\mathfrak{a}}^{\tilde{\mathfrak{q}}}(\lambda_{\mathfrak{q}})=(s_{\Phi'}(t_0)^{-1}\lambda_{\mathfrak{q}})^{\sigma_{\mathfrak{m}}^{-1}}$, where $\Phi':=\Phi^{\tilde{\mathfrak{q}}}$. Let $*$ be the operation introduced in \cite[\S3]{Hay79}. Then, $\Phi'=\mathfrak{m}*\Phi$, as proved in \cite[Proposition 5.10]{Hay79} or \cite[Theorem 8.14]{Hay92} or the proof of \cite[Theorem 5.1]{Hay85}. By \cite[formula (4.5)]{Hay85} we have $s_{\Phi'}(t_0)=s_\Phi(t_0)^{\sigma_{\mathfrak{m}}}$. This proves the lemma.  
\end{proof}
Let us keep the sign-function $\mathbf{sgn}$ fixed. Then, for any nonzero proper ideal $\mathfrak{m}$ of $\mathcal{O}_K$ and any $\eta\in\mathcal{F}_{k_{\mathfrak{m}}}$ the element $(\lambda_{\mathfrak{m}})^\eta$ belongs to $H_{\mathfrak{m}}$, thanks to \cite[Corollary 4.14]{Hay85}. Moreover, by (\ref{stark1}) and \cite[Lemma 2.5]{Hay85} we see that $\mathcal{P}_K$ is generated by $\mu_K$ and by all the norms
\begin{equation*}
\lambda_{\mathfrak{m}}(\mathfrak{g}):=N_{H_{\mathfrak{m}}/H_{\mathfrak{m}}\cap K}(\lambda_{\mathfrak{m}}^{N(\mathfrak{g})-(\mathfrak{g},\, k_{\mathfrak{m}}/k)}),
\end{equation*}
where $\mathfrak{m}$ and $\mathfrak{g}$ are any nonzero coprime ideals of $\mathcal{O}_k$ such that $\mathfrak{m}\neq\mathcal{O}_k$. Furthermore, the map $\alpha:\mathcal{S}(\mathfrak{mg})\longrightarrow k_\infty^\times$, defined by
\begin{equation*}
\alpha(\mathfrak{a}):=N_{KH_{\mathfrak{ma}}/K(\mathfrak{a})}(\lambda_{\mathfrak{ma}}^{N(\mathfrak{g})-(\mathfrak{g},\, k_{\mathfrak{ma}}/k)}),
\end{equation*}
is an Euler system such that $\alpha(1)=\lambda_{\mathfrak{m}}(\mathfrak{g})$. Indeed, the properties E1 and E2 are immediately seen to be satisfied. The property E3 is a consequence of Lemma \ref{unlemme}. The property E4 follows from Lemma \ref{poisson} and \cite[Lemma 4.19]{Hay85}.
\begin{corollary}\label{sardine}If $u\in\mathcal{E}_K$ then there exist an ideal $\mathfrak{g}$ of $\mathcal{O}_k$ and an Euler system $\alpha:\mathcal{S}(\mathfrak{g})\longrightarrow k_\infty^\times$, such that $\alpha(1)=u$
\end{corollary}
\begin{proof} In view of the discussion above we only have to check the corollary for the roots of unity in $K$. But this is obvious.
\end{proof}
\section{The Gras conjecture}
If $\ell\in\mathcal{L}$ then we denote by $\sigma_\ell$ a generator of the cyclic group G$_{\ell}:=\mathrm{Gal}(K(\ell)/K)$. Further, we set
\begin{equation*}
N_\ell:=\sum_{\tau\in\mathrm{G}_{\ell}}\tau\quad\mathrm{and}\quad D_\ell:=\sum_{i=0}^{M-1}i\sigma_\ell.
\end{equation*}
Let $\mathfrak{a}\in\mathcal{S}$ and let G$_{\mathfrak{a}}:=\mathrm{Gal}(K(\mathfrak{a})/K)$. Then $\mathrm{G}_{\mathfrak{a}}\simeq\prod_{\ell\vert\mathfrak{a}}\mathrm{G}_{\ell}$. Moreover, the inertia group of $\ell$ in G$_{\mathfrak{a}}$ is $\mathrm{Gal}(K(\mathfrak{a})/K(\mathfrak{a}/\ell))$, which we shall identify with G$_{\ell}$.\par
Let us now define $Y$ to be the free multiplicative $\mathbb{Z}[\mathrm{G}_{\mathfrak{a}}]$-module generated by the symbols $x(\mathfrak{b})$, $\mathfrak{b}\vert\mathfrak{a}$, and we denote by $Z$ its submodule generated by the relations
\begin{description}
\item{1.} $x(\mathfrak{b})^{\sigma-1}=1,\ \mathrm{for\ all}\ \mathfrak{b}\vert\mathfrak{a}\ \mathrm{and\ all}\ \sigma\in\mathrm{Gal}(K(\mathfrak{a})/K(\mathfrak{b}))$
\item{2.} $x(\mathfrak{b}\ell)^{N_\ell}=x(\mathfrak{b})^{1-\mathrm{Fr}(\ell)^{-1}},\ \mathrm{for\ all}\ \mathfrak{b}\in\mathcal{S}\ \mathrm{and\ all}\ \ell\in\mathcal{L}\ \mathrm{such\ that}\ \mathfrak{b}\ell\vert\mathfrak{a}.$
\end{description}
\begin{lemma}\label{dist} The $\mathbb{Z}[\mathrm{G}_{\mathfrak{a}}]$-module $X:=Y/Z$ is a free $\mathbb{Z}$-module, with basis the set
\begin{equation*}
\lbrace x(\mathfrak{b})^\sigma,\ \mathfrak{b}\vert\mathfrak{a},\sigma\in\mathrm{G}_{\mathfrak{b}}-\cup_{\mathfrak{c}\vert\mathfrak{b},\mathfrak{c}\neq\mathfrak{b}}\mathrm{G}_{\mathfrak{c}}\rbrace.
\end{equation*}
Moreover, if we set $D_{\mathfrak{b}}:=\prod_{\ell\vert\mathfrak{b}} D_\ell$ then, for all $\mathfrak{b}\vert\mathfrak{a}$ and all $\sigma\in\mathrm{G}_{\mathfrak{a}}$, $x(\mathfrak{b})^{D_{\mathfrak{b}}(\sigma-1)}\in X^M$.
\end{lemma}
\begin{proof} The proof is exactly the same as for \cite[Lemma 1.1]{Ru91}.
\end{proof}
\begin{corollary}\label{kappaalpha} For any Euler system $\alpha:\mathcal{S}(\mathfrak{g})\longrightarrow k_\infty^\times$ there is a natural map
\begin{equation}\label{kappa}
\kappa_\alpha:\mathcal{S}(\mathfrak{g})\longrightarrow K^\times/(K^\times)^M,\quad\mathrm{such\ that}\quad \kappa_\alpha(\mathfrak{a})\equiv\alpha(\mathfrak{a})^{D_{\mathfrak{a}}}\ \mathrm{modulo}\ (K^\times)^M.   
\end{equation} 
\end{corollary}
\begin{proof} The existence of $\kappa_\alpha$ is deduced from the above lemma \ref{dist} by applying the theorem 90 of Hilbert. We refer the reader to \cite[Proposition 2.2]{Ru91}.
\end{proof}
To go further we need to understand the prime factorization of $\kappa_\alpha(\mathfrak{a})$, for $\mathfrak{a}\in\mathcal{S}(\mathfrak{g})$. To this end we adopt the following notation of Rubin, cf.\cite[\S2]{Ru91}. Let $\mathcal{I}=\oplus_\lambda\mathbb{Z}\lambda$ be the group of fractional ideals of $\mathcal{O}_K$ written additively. If $\ell$ is a prime ideal of $\mathcal{O}_k$ then we write $\mathcal{I}_\ell=\oplus_{\lambda\vert\ell}\mathbb{Z}\lambda$. If $y\in K^\times$ then we denote by $(y)_\ell\in\mathcal{I}_\ell$, $[y]\in\mathcal{I}/M\mathcal{I}$ and $[y]_\ell\in\mathcal{I}_\ell/M\mathcal{I}_\ell$ the projections of the fractional ideal $(y):=y\mathcal{O}_K$.
\begin{proposition}\label{fielle1} Let $\ell\in\mathcal{L}$ and let us consider the map
\begin{equation}\label{requin}
\psi_\ell:K(\ell)^\times\longrightarrow(\mathcal{O}_K/\ell\mathcal{O}_K)^\times/((\mathcal{O}_K/\ell\mathcal{O}_K)^\times)^M,   
\end{equation}
which associates to $z$ the sum $\oplus_{\lambda\vert\ell}z_\lambda$ such that the image of $z^{1-\sigma_\ell}$ in $(\mathcal{O}_K/\lambda)^\times$ is equal to $(z_\lambda)^d$, where $d:=(N(\ell)-1)/M$. Then there exists a unique $\mathrm{G}$-equivariant isomorphism
\begin{equation}\label{fielle}
\varphi_\ell:(\mathcal{O}_K/\ell\mathcal{O}_K)^\times/((\mathcal{O}_K/\ell\mathcal{O}_K)^\times)^M\longrightarrow\mathcal{I}_\ell/M\mathcal{I}_\ell,   
\end{equation}
such that
\begin{equation}\label{baleine}
(\varphi_\ell\circ\psi_\ell)(x)=[N_{K(\ell)/K}(x)]_\ell.   
\end{equation}
\end{proposition}
\begin{proof} We first prove the existence of $\varphi_\ell$. Let $\lambda'$ be a prime ideal of $\mathcal{O}_{K(\ell)}$ above $\ell$. Let $v_{\lambda'}$ be the normalized valuation of $K(\ell)$ defined by $\lambda'$, and let $\pi\in\lambda'-(\lambda')^2$. Then $\pi^{1-\sigma_\ell}$ has exact order $M$ in the cyclic group $(\mathcal{O}_{K(\ell)}/\lambda')^\times$, because $K(\ell)/K$ is cyclic, totally ramified at $\lambda:=\lambda'\cap\mathcal{O}_K$. In particular, using the isomorphism $\mathcal{O}_{K(\ell)}/\lambda'\simeq\mathcal{O}_K/\lambda$, there exists $x_\lambda\in(\mathcal{O}_K/\lambda)^\times$ such that the image of $\pi^{1-\sigma_\ell}$ in $(\mathcal{O}_K/\lambda)^\times$ is equal to $(x_\lambda)^d$. Let us remark that the projection of $x_\lambda$ in $(\mathcal{O}_K/\lambda)^\times/((\mathcal{O}_K/\lambda)^\times)^M$ is well defined, does not depend on $\pi$ and, in fact, has exact order $M$. Thus, the isomorphism $\mathcal{O}_K/\ell\mathcal{O}_K\simeq\oplus_{\lambda\vert\ell}\mathcal{O}_K/\lambda$ allows us to define a $\mathrm{G}$-equivariant isomorphism
\begin{equation*}
\hat{\varphi}_\ell:(\mathcal{O}_K/\ell\mathcal{O}_K)^\times/((\mathcal{O}_K/\ell\mathcal{O}_K)^\times)^M\longrightarrow\mathcal{I}_\ell/M\mathcal{I}_\ell,   
\end{equation*}
such that the image of an element $x:=\oplus_{\lambda\vert\ell}(x_\lambda)^{e_\lambda}$ is $\hat{\varphi}_\ell(x):=\oplus_{\lambda\vert\ell}e_\lambda\lambda$.
It is clear that the map $\varphi_\ell:=-\hat{\varphi}_\ell$ satisfies (\ref{baleine}). The unicity of $\varphi_\ell$ follows from the fact that $\psi_\ell$ and the map $K(\ell)^\times\longrightarrow\mathcal{I}_\ell/M\mathcal{I}_\ell$, $x\longmapsto[N_{K(\ell)/K}(x)]_\ell$ have the same kernel, that is the set of elements $x\in K(\ell)^\times$ such that $v_{\lambda'}(x)\equiv0$ modulo $M$, for all prime ideal $\lambda'$ above $\ell$.
\end{proof}
The map $\varphi_\ell$ induces a homomorphism $\{y\in K^\times/(K^\times)^M,\ [y]_\ell=0\}\longrightarrow\mathcal{I}_\ell/M\mathcal{I}_\ell$ which we also denote by $\varphi_\ell$. 
\begin{lemma} For any Euler system $\alpha:\mathcal{S}(\mathfrak{g})\longrightarrow k_\infty^\times$, and any $\mathfrak{a}\in\mathcal{S}(\mathfrak{g})$, such that $\mathfrak{a}\neq1$
\begin{equation}\label{marrakech}
[\kappa_\alpha(\mathfrak{a})]_\ell=
\begin{cases}
0&\mathrm{if}\ \ell\nmid\mathfrak{a}\\
\varphi_\ell(\kappa_\alpha(\mathfrak{a}/\ell))&\mathrm{if}\ \ell\vert\mathfrak{a}.
\end{cases}
\end{equation}
\end{lemma}
\begin{proof} The proof is exactly the same as in \cite[Proposition 2.4]{Ru91}.
\end{proof}
In the sequel, if $p$ is a prime number such that $p\nmid[K:k]$, $\chi$ a nontrivial irreducible $\mathbb{Z}_p$-character of $\mathrm{G}$, and $\Pi$ is a $\mathbb{Z}_p[\mathrm{G}]$-module then we define $\Pi_\chi:=e_\chi\Pi$. If $\Pi$ is a $\mathbb{Z}[\mathrm{G}]$-module then we define $\Pi_\chi:=e_\chi(\mathbb{Z}_p\otimes_{\mathbb{Z}}\Pi)$. Before proving Theorem \ref{tresgras} we first need to prove the analoguous of \cite[Theorem 4]{Ru94} and \cite[Theorem 3.1]{Ru91}. For this we set
\begin{equation*}
K':=K_M(a_1^{1/M},\ldots,a_s^{1/M}).
\end{equation*}
\begin{lemma}\label{fund1} Let $p$ be a prime number such that $p\nmid\rho[K:k]$ and let $M$ be a power of $p$. Then 
the natural map
\begin{equation*}
f:K^\times/(K^\times)^M\longrightarrow K_M^\times/(K_M^\times)^M
\end{equation*}
is injective if $\mu_p\not\subset k$. But if $\mu_p\subset k$ its kernel is contained in $\mathbb{F}_{q}^{\times}/(\mathbb{F}_{q}^{\times})^M$. In particular, it is annihilated by $[K:k]-s(\mathrm{G})$, where $s(\mathrm{G}):=\sum\sigma, \sigma\in\mathrm{G}$. Furthermore, the kernel of the natural map
\begin{equation*}
g:K^\times/(K^\times)^M\longrightarrow K'^\times/(K'^\times)^M
\end{equation*}
is also annihilated by $[K:k]-s(\mathrm{G})$.
\end{lemma}
\begin{proof} Let $\mathbb{F}_{q^r}$ (resp. $\mathbb{F}_{q^{r'}}$) be the field of constants of $K$ (resp. $K_M$). Then $K_M=K\mathbb{F}_{q^{r'}}$ and $K_M$ is a cyclic extension of $K$, with Galois group canonically isomorphic to $\mathrm{Gal}(\mathbb{F}_{q^{r'}}/\mathbb{F}_{q^r})$. Since $H^1(\mathbb{F}_{q^{r'}}/\mathbb{F}_{q^r},\mathbb{F}_{q^{r'}}^\times)=0$ the kernel of $V$ is equal to 
\begin{equation*}
(\mathbb{F}_{q^r}^\times\cap(\mathbb{F}_{q^{r'}}^\times)^M)/(\mathbb{F}_{q^r}^\times)^M
\end{equation*}
If $\mu_p\not\subset k$ then $\mathbb{F}_{q^{r'}}=\mathbb{F}_{q^r}(\mu_M)$. In this case it is easy to check that $\mathbb{F}_{q^r}^\times\cap(\mathbb{F}_{q^{r'}}^\times)^M=(\mathbb{F}_{q^r}^\times)^M$. If $\mu_p\subset k$ then $\mathbb{F}_{q^{r'}}=\mathbb{F}_{q^r}((\mathbb{F}_q^\times)^{1/M})$. In particular $(\mathbb{F}_{q^{r'}}^\times)^M=\mathbb{F}_{q^r}^\times$. Moreover, since $p\nmid[K:k]$ the integer $r$ is prime to $p$. and hence $\mathbb{F}_{q^r}^\times/(\mathbb{F}_{q^r}^\times)^M$ is naturally isomorphic to $\mathbb{F}_q^\times/(\mathbb{F}_q^\times)^M$. This proves the assertions of the lemme about the kernel of $f$. Further, since $K'/K_M$ is a Kummer extension and $a_1,\ldots,a_s$ are elements of $k$, an elementary use of Kummer theory shows that the kernel of $g$ is also annihilated by $[K:k]-s(\mathrm{G})$. 
\end{proof}
\begin{lemma}\label{capital} Suppose $p$ is a prime number such that $p\nmid\rho[K:k]$ and let $M$ be a power of $p$. Let $\chi$ be a nontrivial irreducible $\mathbb{Z}_p$-character of $\mathrm{G}$. If $\mu_p\subset K$ and $p\,\vert[H:k]$, then we assume $\chi\neq\omega$. Let $\mathsf{H}^\chi$ be the abelian extension of $K$ corresponding to the $\chi$-part $\mathrm{Cl}(\mathcal{O}_K)_\chi$. Then $\mathsf{H}^\chi\cap K'=K$.
\end{lemma}
\begin{proof} The group $\mathrm{G}$ acts trivially on $\mathrm{Gal}(\mathsf{H}^\chi\cap K_M/K)$ because $K_M$ is abelian over $k$. On the other hand, $\mathrm{Gal}(\mathsf{H}^\chi\cap K_M/K)$ is a $\mathrm{G}$-quotient of $\mathrm{Gal}(\mathsf{H}^\chi/K)\simeq\mathrm{Cl}(\mathcal{O}_K)_\chi$. This implies that $\mathsf{H}^\chi\cap K_M=K$ since $\chi\neq1$. In addition, if $p\nmid[H:k]$ then $K'=K_M$. In particular we have proved that $\mathsf{H}^\chi\cap K'=K$ in case $p\nmid[H:k]$. Let $E:=K_M(\mathsf{H}^\chi\cap K')$. By Kummer theory we deduce from the inclusion $E\subset K'$ that $E=K_M(V^{1/M})$, where $V$ is a subgroup of the multiplicative group $<a_1,\ldots,a_s>\subset k^\times$. But recall that $K_M$ and $\mathsf{H}^\chi$ are abelian over $K$. In particular, if $x\in V^{1/M}$ and $\tau\in\mathrm{Gal}(E/K_M)$ we have $\tau(x)/x\in\mu_M\cap K$. Thus, if $\mu_p\not\subset K$ then $E=K_M$ and $\mathsf{H}^\chi\cap K'=K$ because of the isomorphism $\mathrm{Gal}(E/K_M)\simeq\mathrm{Gal}(\mathsf{H}^\chi\cap K'/K)$. If $\mu_p\subset K$, $p\,\vert[H:k]$ then $\mathrm{G}$ acts on $\mathrm{Gal}(E/K_M)$ via $\omega$. This implies that $\mathrm{Gal}(E/K_M)=1$ because this group is isomorphic to $\mathrm{Gal}(\mathsf{H}^\chi\cap K'/K)$ on which $\mathrm{G}$ acts via $\chi\neq\omega$. The proof of the lemma is now complete.
\end{proof}
\begin{theorem}\label{Chebotarev} Suppose $p$ is a prime number such that $p\nmid\rho[K:k]$ and let $M$ be a power of $p$. Let $\chi$ be a nontrivial irreducible $\mathbb{Z}_p$-character of $\mathrm{G}$. If $\mu_p\subset K$ and $p\,\vert[H:k]$, then we assume $\chi\neq\omega$. Let $\beta\in (K^\times/(K^\times)^M)_\chi$ and $A$ be a $\mathbb{Z}_p[\mathrm{G}]$-quotient of $\mathrm{Cl}(\mathcal{O}_K)_\chi$. Let $m$ be the order of $\beta$ in $K^\times/(K^\times)^M$, $W$ the $\mathrm{G}$-submodule of $K^\times/(K^\times)^M$ generated by $\beta$, $\mathsf{H}$ the abelian extension of $K$ corresponding to $A$, and $L:=\mathsf{H}\cap K'(W^{1/M})$. Then, there is a $\mathbb{Z}[\mathrm{G}]$ generator $\mathfrak{c}'$ of $\mathrm{Gal}(L/K)$ such that for any $\mathfrak{c}\in A$ whose restriction to $L$ is $\mathfrak{c}'$, there are infinitely many prime ideals $\lambda$ of $\mathcal{O}_K$ such that:
\begin{description}
\item(i) the projection of the class of $\lambda$ in $A$ is $\mathfrak{c}$,
\item(ii) if $\ell:=\lambda\cap\mathcal{O}_k$ then $\ell\in\mathcal{L}$,
\item(iii) $[\beta]_\ell=0$ and there is $u\in(\mathbb{Z}/M\mathbb{Z}[\mathrm{G}])_\chi^\times$ such that $\varphi_\ell(\beta)=(M/m)u\lambda$.
\end{description} 
\end{theorem}
\begin{proof} We follow \cite[Theorem 3.1]{Ru91}. Since $W\subset (K^\times/(K^\times)^M)_\chi$ and $\chi\neq1$, we deduce from Lemma \ref{fund1} that the Galois group of the Kummer extension $K'(W^{1/M})/K'$ is isomorphic as a $\mathbb{Z}[\mathrm{Gal}(K_M/k)]$-module to $\mathrm{Hom}(W,\mu_M)$. But $W\simeq(\mathbb{Z}/m\mathbb{Z}[\mathrm{G}])_\chi$, which is a direct factor of $(\mathbb{Z}/m\mathbb{Z})[\mathrm{G}]$.
On the other hand, $\mathrm{Hom}((\mathbb{Z}/m\mathbb{Z})[\mathrm{G}],\mu_M)$ is $\mathbb{Z}[\mathrm{Gal}(K_M/k)]$-cyclic, generated for instance by the group homomorphism $\Psi:(\mathbb{Z}/m\mathbb{Z})[\mathrm{G}]\longrightarrow\mu_M$ defined by $\Psi(1)=\zeta$ and $\Psi(g)=1$, for $g\neq1$, where $\zeta\in\mu_M$ is a primitive $m$-th root of unity. Therefore, we can find $\tau\in\mathrm{Gal}(K'(W^{1/M})/K')$ which generates $\mathrm{Gal}(K'(W^{1/M})/K')$ over
$\mathbb{Z}[\mathrm{Gal}(K_M/k)]$. The restriction $\mathfrak{c}'$ of $\tau$ to $L$ is a  $\mathbb{Z}[\mathrm{G}]$ generator of $\mathrm{Gal}(L/K)\simeq\mathrm{Gal}(LK'/K')$ by Lemma \ref{capital}. Let $\mathfrak{c}\in\mathrm{Gal}(\mathsf{H}/K)=A$ be any extension of $\mathfrak{c}'$ to $\mathsf{H}$. Then one can find $\sigma\in\mathrm{Gal}(\mathsf{H}K'(W^{1/M})/K)$ such that
\begin{equation*}
\sigma_{\vert\mathsf{H}}=\mathfrak{c}\quad\mathrm{and}\quad\sigma_{\vert K'(W^{1/M})}=\tau.
\end{equation*}
By \cite[Theorem 12 of Chapter XIII, page 289]{Weil} there exist infinitely many primes $\lambda$ of $\mathcal{O}_K$ whose Frobenius in $\mathrm{Gal}(\mathsf{H}K'(W^{1/M})/K)$ is the congugacy class of $\sigma$, and such that $\ell:=\lambda\cap\mathcal{O}_k$ is unramified in $K'(W^{1/M})/k$. Now it is immediate that $(i)$ and $(ii)$ are satisfied. The rest of the proof is exactly the same as the proof of \cite[Theorem 3.1]{Ru91}.
\end{proof}
\begin{theorem}\label{casfacile} Suppose $p$ is a prime number such that $p\nmid\rho[K:k]$. Let $\chi$ be a nontrivial irreducible $\mathbb{Z}_p$-character of $\mathrm{G}$. If $\mu_p\in K$ and $p\,\vert[H:k]$, then we assume $\chi\neq\omega$. Then we have
\begin{equation}\label{division}
\#\mathrm{Cl}(\mathcal{O}_K)_\chi\,\vert\,\#(\mathcal{O}_K^\times/\mathcal{E}_K)_\chi.   
\end{equation}
\end{theorem}
\begin{proof} We follow \cite[Theorem 3.2]{Ru91}. Let $\hat{\chi}$ be a $p$-adic irreducible character of $\mathrm{G}$ such that $\hat{\chi}\vert\chi$, and let $\hat{\chi}(\mathrm{G}):=\{\hat{\chi}(\sigma),\ \sigma\in\mathrm{G}\}$. Then, the ring $R:=\mathbb{Z}_p[\mathrm{G}]_\chi$ is isomorphic to $\mathbb{Z}_p[\hat{\chi}(\mathrm{G})]$, which is the ring of integers of the unramified extension $\mathbb{Q}_p[\hat{\chi}(\mathrm{G})]$ of $\mathbb{Q}_p$. Thus, $R$ is a discret valuation ring. Moreover, the $R$-torsion of any $R$-module is equal to its $\mathbb{Z}_p$-torsion. It is well known that $\mathcal{O}_K^\times/\mu_K$ is a free $\mathbb{Z}$-module of rank $[K:k]-1$. More precisely there exists an isomorphism of $\mathbb{Z}[\mathrm{G}]$-modules $\log:\mathcal{O}_K^\times/\mu_K\longrightarrow\Delta$, where $\Delta$ is a submodule of the augmentation ideal of $\mathbb{Z}[\mathrm{G}]$. Moreover, $\Delta$ is a free $\mathbb{Z}$-module of rank $[K:k]-1$. Hence, since $\chi\neq1$, the quotient $(\mathcal{O}_K^\times)_\chi/(\mu_K)_\chi$ is a free $R$-module of rank $1$. Let us define
\begin{equation*}
M:=p\#(\mathcal{O}_K^\times/\mathcal{E}_K)_\chi\#\mathrm{Cl}(\mathcal{O}_K)_\chi.
\end{equation*}
Let $\mu$, $U$ and $V$ be the images of $\mu_K$, $\mathcal{O}_K^\times$ and $\mathcal{E}_K$ in $K^\times/(K^\times)^M$ respectively. We deduce from above that, $U_\chi/\mu_\chi$ is a free $R/MR$-module of rank $1$. But since
\begin{equation}\label{tachfine}
U_\chi/V_\chi\simeq(\mathcal{O}_K^\times)_\chi/(\mathcal{E}_K)_\chi\simeq R/tR,
\end{equation}
for some divisor $t$ of $M$, there exists $\xi\in U_\chi$ giving an $R$-basis of $U_\chi/\mu_\chi$ and such that $\xi^t\in(\mathcal{E}_K)_\chi$. In particular $\xi$ has order $M$ in $K^\times/(K^\times)^M$. By Corollary \ref{sardine} there exist an ideal $\mathfrak{g}$ of $\mathcal{O}_k$ and an Euler system $\alpha:\mathcal{S}(\mathfrak{g})\longrightarrow k_\infty^\times$, such that the map $\kappa:=\kappa_\alpha$ defined by (\ref{kappa}) satisfies $\kappa(1)=\xi^t$. We define inductively classes $\mathfrak{c}_0,\ldots,\mathfrak{c}_i\in\mathrm{Cl}(\mathcal{O}_K)_\chi$, prime ideals $\lambda_1,\ldots,\lambda_i$ of $\mathcal{O}_K$, coprime with $\mathfrak{g}$, and ideals $\mathfrak{a}_0,\ldots,\mathfrak{a}_i$ of $\mathcal{O}_k$ such that $\mathfrak{c}_0=1$ and $\mathfrak{a}_0=1$. Let $i\geq0$, and suppose that $\mathfrak{c}_0,\ldots,\mathfrak{c}_i$ and $\lambda_1,\ldots,\lambda_i$ (if $i\geq1$) are defined. Then we set $\mathfrak{a}_i=\prod_{1\leq n\leq i}\ell_n$ (if $i\geq1$), where $\ell_n:=\lambda_n\cap\mathcal{O}_k$. Moreover,
\begin{itemize}
\item If $\mathrm{Cl}(\mathcal{O}_K)_\chi\neq<\mathfrak{c}_0,\ldots,\mathfrak{c}_i>_{\mathrm{G}}$, where $<\mathfrak{c}_0,\ldots,\mathfrak{c}_i>_{\mathrm{G}}$ is the G-module generated by $\mathfrak{c}_0,\ldots,\mathfrak{c}_i$, then we define $\mathfrak{c}_{i+1}$ to be any element of $\mathrm{Cl}(\mathcal{O}_K)_\chi$ whose image in $\mathrm{Cl}(\mathcal{O}_K)_\chi/<\mathfrak{c}_0,\ldots,\mathfrak{c}_i>_{\mathrm{G}}$ is nontrivial and is equal to a class $\mathfrak{c}$ which restricts to the generator $\mathfrak{c}'$ of $\mathrm{Gal}(L/K)$ in Theorem \ref{Chebotarev} applied to $\beta:=\kappa(\mathfrak{a}_i)_\chi$, the image of $\kappa(\mathfrak{a}_i)$ in $(K^\times/(K^\times)^M)_\chi$, and $A:=\mathrm{Cl}(\mathcal{O}_K)_\chi/<\mathfrak{c}_0,\ldots,\mathfrak{c}_i>_{\mathrm{G}}$. Also we let $\lambda_{i+1}$ be any prime ideal of $\mathcal{O}_K$ prime to $\mathfrak{g}$ and satisfying Theorem \ref{Chebotarev} with the same conditions.
\item If $\mathrm{Cl}(\mathcal{O}_K)_\chi=<\mathfrak{c}_0,\ldots,\mathfrak{c}_i>_{\mathrm{G}}$ then we stop.
\end{itemize}
This construction of our classes $\mathfrak{c}_j$ implies that the ideals $\ell_j:=\lambda_j\cap\mathcal{O}_k\in\mathcal{S}(\mathfrak{g})$. Let $m_i$ be the order of $\kappa(\mathfrak{a}_i)_\chi$ in $K^\times/(K^\times)^M$, and let $t_i:=M/m_i$. By the assertion $(iii)$ of Theorem \ref{Chebotarev} we have $\varphi_{\ell_{i+1}}(\kappa(\mathfrak{a}_i)_\chi)=ut_i\lambda_{i+1}$, for some $u\in\ \mathbb{Z}/M\mathbb{Z}[\mathrm{G}]_\chi^\times$. But $\mathfrak{a}_{i+1}=\mathfrak{a}_i\ell_{i+1}$. Thus 
\begin{equation}\label{thonrouge}
[\kappa(\mathfrak{a}_{i+1})_\chi]_{\ell_{i+1}}=ut_i\lambda_{i+1},
\end{equation}
thanks to (\ref{marrakech}). Now, by the definition of $t_{i+1}$, the fractional ideal of $\mathcal{O}_K$ generated by $\kappa(\mathfrak{a}_{i+1})_\chi$ is a $t_{i+1}$-th power. Thus, we must have $t_{i+1}\vert t_i$. Actually, we can say more. Indeed, there exist $\zeta\in\mu_K$ and $z\in K^\times$ such that $\kappa(\mathfrak{a}_{i+1})_\chi=\zeta z^{t_{i+1}}$. Therefore, (\ref{marrakech}) and (\ref{thonrouge}) imply
\begin{equation}\label{toumarte}
z\mathcal{O}_K=(\lambda_{i+1})^{ut_i/t_{i+1}}(\prod_{j=1}^i\lambda_j^{u_j})\mathfrak{b}^{M/t_{i+1}},
\end{equation}
where $u_j\in\mathbb{Z}_p[\mathrm{G}]$ for all $j\in\{1,\ldots,i\}$ and $\mathfrak{b}$ is a fractional ideal of $\mathcal{O}_K$. But we see from (\ref{tachfine}) that $t_0\vert\#(\mathcal{O}_K^\times/\mathcal{E}_K)_\chi$, and since $t_{i+1}\vert t_0$ the integer $M/t_{i+1}$ annihilates $\mathrm{Cl}(\mathcal{O}_K)_\chi$. The identity (\ref{toumarte}) then implies
\begin{equation}\label{idriss}
(t_i/t_{i+1})\mathfrak{c}_{i+1}\in<\mathfrak{c}_0,\ldots,\mathfrak{c}_i>_{\mathrm{G}}.
\end{equation}
Let $\dim(\chi):=[\mathbb{Q}_p[\hat{\chi}(\mathrm{G})]:\mathbb{Q}_p]$, then (\ref{idriss}) implies
\begin{equation*}
\#\mathrm{Cl}(\mathcal{O}_K)_\chi\,\vert\,\prod_{j=1}^i(t_{j-1}/t_j)^{\dim(\chi)}\,\vert\, t_0^{\dim(\chi)}=\#(\mathcal{O}_K^\times/\mathcal{E}_K)_\chi.
\end{equation*} 
\end{proof}
\noindent{\bf Proof of Theorem \ref{tresgras}}. Let the hypotheses and notation be as in Theorem \ref{tresgras}. Let $\Psi$ be the irreducible rational character of $\mathrm{G}$ such that $\chi\vert\Psi$. The formula (\ref{indice2}) may be written as follows
\begin{equation*}
\prod_{\chi'\vert\Psi}\#\mathrm{Cl}(\mathcal{O}_K)_{\chi'}=\prod_{\chi'\vert\Psi}\#(\mathcal{O}_K^\times/\mathcal{E}_K)_{\chi'},
\end{equation*}
where $\chi'$ runs over the irreducible $\mathbb{Z}_p$-characters of $\mathrm{G}$ such that $\chi'\vert\Psi$. Moreover, the formula (\ref{division}) is satisfied for such characters $\chi'$ since $\chi\not\in\Xi_p$. This implies (\ref{gras}). \par
\vskip 5pt
\noindent{\bf Acknowledgement} The authors express their sincere thanks to the referee for his suggestions to improve the content of the paper.

\def\cprime{$'$} \def\cprime{$'$} \def\cprime{$'$}

\address{Laboratoire de math\'ematique\\ 16
Route de Gray\\ 25030 Besan\c con cedex\\ France}
{houkhaba@univ-fcomte.fr\\ sviguie@univ-fcomte.fr}

\end{document}